\documentclass[12pt]{amsart}
\usepackage[letterpaper,lmargin=1.5in,height=8.3in,width=6in,top=1.5in]{geometry}
\usepackage{amsmath}
\usepackage{amsthm}
\usepackage{amssymb}
\newcommand{\NN}{\mathbb{N}}

\newcommand{\RR}{\mathbb{R}}

\newcommand{\mb}[1]{\mathbf{#1}}
\newcommand{\mc}[1]{\mathcal{#1}}

\newcommand{\Graph}{\mathbf{Graph}}
\newcommand{\Gr}{\mathbf{Gr}}
\newcommand{\Aut}{\mathrm{Aut}}
\newcommand{\GG}{\mathcal{G}}
\newcommand{\LL}{\mathcal{L}}
\newcommand{\gr}{\mathrm{gr}}
\newcommand{\Cay}{\mathrm{Cay}}

\newcommand{\e}{\varepsilon}

\newtheorem{theo}{Theorem}[section]
\newtheorem{lem}[theo]{Lemma}
\newtheorem{prop}[theo]{Proposition}
\newtheorem{cor}[theo]{Corollary}

\theoremstyle{definition}
\newtheorem{defn}[theo]{Definition}
\newtheorem{oq}[theo]{Open Question}

\begin{document}
\title{Graphings and unimodularity}
\author[I. Artemenko]{Igor Artemenko}
\address{Department of Mathematics and Statistics, University of Ottawa, 585 King Edward Ave., Ottawa, Ontario, Canada K1N 6N5}

\begin{abstract}
We extend the concept of the law of a finite graph to graphings, which are, in general, infinite graphs whose vertices are equipped with the structure of a probability space. By doing this, we obtain a vast array of new unimodular measures. Furthermore, we work out in full detail a proof of a known result, which states that weak limits preserve unimodularity.
\end{abstract}

\maketitle

%
%
\section*{Introduction}

This article looks at graphs from the viewpoint of probability theory by defining measures on the space of rooted graphs. We are concerned with approximating such measures using finite graphs. More precisely, every finite graph $G$ gives rise to a probability measure known as the law of $G$, and approximations are done by means of weak convergence of sequences of laws.

Unimodularity is a property of probability measures, which is known to be preserved under weak limits. Although this result has been stated by Aldous and Lyons \cite{pourn}, Schramm \cite{hgl}, and Elek \cite{otlolggs}, we begin the article by giving a detailed argument. Following that, we expose an abundant source of examples of unimodular measures using graphings, which are graphs whose vertices support the structure of a probability space.

There are several important open questions that are related to unimodularity. The primary question, brought up by David Aldous and Russell Lyons \cite{pourn}, is whether every unimodular measure can be approximated by laws of finite graphs. This problem can be decomposed into the following questions:
\begin{enumerate}
\item[(i)]
Can the law of a graphing be approximated by laws of finite graphs?
\item[(ii)]
Is every unimodular measure the law of a graphing?
\end{enumerate}

Many of the concepts are introduced without examples, and the reader is encouraged to see this author's previous work \cite{wcolofg} for a more thorough treatment of the basics. However, note that the notation used here is different.

G\'abor Elek discusses some of the material in this article as well \cite{nolofg,otlolggs}, but proceeds in a slightly different direction. In fact, the notation we use mimics his.

To be consistent, note the following set of guidelines regarding notation and convention. All graphs are assumed to be simple and undirected. Throughout the article, assume that $X$ is a compact metric space. Denote by $\mc{M}(X)$ the set of probability measures on $X$, by $\mb{C}(X)$ the set of continuous real-valued functions on $X$, and by $\mc{B}(X)$ the Borel $\sigma$-algebra on $X$. From now on, the reader may assume that all of our measures are probability measures.

If $d$ is a metric on $X$, then $B_d(x,r) = \{y \in X ~:~ d(x,y) \leq r\}$ is the ball around $x$ of radius $r$. The set $\gr(f) = \{(x,f(x)) ~:~ x \in X\}$ is the graph of a function $f : X \to Y$.

\subsection*{Acknowledgements}
This article is based on the research done in the Summer of 2011 under the supervision of Dr. Vladimir Pestov, funded by the NSERC USRA. Special thanks go out to Dr. Vadim Kaimanovich for the fruitful discussions.
%
%
\section{Measures and metrics}

We begin by introducing the basic concepts that are used throughout this article.

\begin{defn}
A sequence $(\mu_n)_{n=1}^\infty$ of measures on $X$ \emph{converges weakly} to some $\mu \in \mc{M}(X)$ if
\[
\int f ~d\mu_n \to \int f ~d\mu
\]
for all $f \in \mb{C}(X)$. The measure $\mu$ is known as the \emph{weak limit} of the given sequence.

If $f : X \to Y$ is a measurable function between the measure spaces $X$ and $Y$, the \emph{pushforward of $\mu$} is the measure $f_\ast(\mu)$ on $Y$ defined by
\[
f_\ast(\mu)(B) = \mu(f^{-1}(B))
\]
for all measurable subsets $B$ of $Y$.
\end{defn}

\begin{prop}\label{pushforward_is_continuous}
Let $X$ and $Y$ be compact metric spaces. Suppose that $f : X \to Y$ is a continuous function. If $(\mu_n)_{n=1}^\infty$ is a sequence of measures on $X$ that converges weakly to some $\mu \in \mc{M}(X)$, then $(f_\ast(\mu_n))_{n=1}^\infty$ converges weakly to $f_\ast(\mu)$.
\end{prop}

\begin{proof}
Let $(\mu_n)_{n=1}^\infty$ be a sequence of measures on $X$ that converges weakly to $\mu$. Suppose that $g \in \mb{C}(Y)$. Then
\[
\int g ~df_\ast(\mu_n) = \int (g \circ f) ~d\mu_n \to \int (g \circ f) ~d\mu = \int g ~df_\ast(\mu)
\]
because the composition $g \circ f$ is continuous.
\end{proof}

\begin{prop}\label{intersection_in_ultrametric}
Let $(X,d)$ be an ultrametric space. If $r \leq s$ and $B_d(x,r) \cap B_d(y,s)$ is nonempty, then $B_d(x,r) \subseteq B_d(y,s)$.
\end{prop}

\begin{proof}
Suppose that $z$ lies in the intersection, which means $d(x,z) \leq r$ and $d(y,z) \leq s$. If $w \in B_d(x,r)$, then
\[
d(y,w) \leq \max\{d(y,z),d(z,w)\} \leq \max\{d(y,z),d(z,x),d(x,w)\} \leq s
\]
because $d(x,w) \leq r$ and $d$ is an ultrametric, and so $w \in B_d(y,s)$.
\end{proof}

\begin{cor}\label{balls_are_closed_open}
A ball of nonzero radius in an ultrametric space $(X,d)$ is closed and open.
\end{cor}

\begin{proof}
Consider the ball $B = B_d(x,r)$ for some $x \in X$ and $r$ a positive real number. By definition, $B$ is closed. To see that $B$ is open, let $y \in B$. Since $B_d(x,r) \cap B_d(y,r)$ is nonempty, Proposition \ref{intersection_in_ultrametric} implies that $B_d(x,r) = B_d(y,r)$. Then
\[
\{z \in X ~:~ d(y,z) < r\} \subseteq B_d(y,r) = B_d(x,r) = B,
\]
and so $B$ is open.
\end{proof}

\begin{lem}\label{simple_functions_approximate_continuous_functions}
Let $(X,d)$ be a compact ultrametric space, and let $\mu$ be a measure on $X$. If $f \in \mb{C}(X)$ and $\e$ is a positive real number, there is a simple function
\[
s_\e = \sum_{i=1}^k a_i \chi_{B_i}
\]
for some real numbers $a_i$ and balls $B_i$ such that $\left|\int (f - s_\e) ~d\mu\right| < \e$. The function $s_\e$ does not depend on the measure $\mu$.
\end{lem}

\begin{proof}
Let $\e$ be a positive real number. Since $X$ is compact, the function $f$ is uniformly continuous, which means there is a positive real number $\delta$ such that
\[
\forall x \in X ~~ \forall y \in X ~~ d(x,y) < \delta ~~ \Rightarrow ~~ |f(x) - f(y)| < \e.
\]
Furthermore, the set $X$ can be covered by the collection $\{B_d(x,\delta) ~:~ x \in X\}$ of open sets. Using the fact that $X$ is compact, it follows that
\[
X = \bigcup_{i=1}^k B_d(x_i,\delta)
\]
for some $x_i \in X$. This union is disjoint because $X$ is an ultrametric space. Consider the function
\[
s_\e = \sum_{i=1}^k f(x_i) \chi_{B_i}
\]
where $B_i = B_d(x_i,\delta)$. Following this, if $x \in X$, there is a unique integer $i$ such that $x \in B_d(x_i,\delta)$. Then
\[
|f(x) - s_\e(x)| = |f(x) - f(x_i)| < \e
\]
because $d(x,x_i) < \delta$. Thus $|f(x) - s_\e(x)| < \e$ for all $x \in X$. By several properties of integration, we see that
\[
\left|\int (f - s_\e) ~d\mu\right| \leq \int |f - s_\e| ~d\mu < \e,
\]
as required.
\end{proof}

\begin{theo}\label{weak_conv_iff_char_func_conv}
Let $(X,d)$ be a compact ultrametric space. A sequence $(\mu_n)_{n=1}^\infty$ of measures on $X$ converges weakly to $\mu \in \mc{M}(X)$ if and only if
\[
\forall \e > 0 ~~ \forall x \in X ~~ \mu_n(B_d(x,\e)) \to \mu(B_d(x,\e)).
\]
\end{theo}

\begin{proof}
Let $B = B_d(x,\e)$ for some $x \in X$ and $\e$ a positive real number. Since $X$ is an ultrametric space, the set $B$ is closed and open by Corollary \ref{balls_are_closed_open}, which means the characteristic function $\chi_B$ is continuous on $X$. If $\mu$ is the weak limit of $(\mu_n)_{n=1}^\infty$, then
\[
\mu_n(B) = \int \chi_B ~d\mu_n \to \int \chi_B ~d\mu = \mu(B).
\]

Conversely, to see that the sequence $(\mu_n)_{n=1}^\infty$ converges weakly to $\mu$, let $f \in \mb{C}(X)$, and let $\e$ be a positive real number. By Lemma \ref{simple_functions_approximate_continuous_functions}, there is a simple function
\[
s_\e = \sum_{i=1}^k a_i \chi_{B_i}
\]
such that $|\int (f - s_\e) ~d\mu| < \e$ and $|\int (f - s_\e) ~d\mu_n| < \e$ for all positive integers $n$. By the hypothesis and the linearity of integration,
\[
\int s_\e ~d\mu_n \to \int s_\e ~d\mu,
\]
so there exists a positive integer $N$ such that
\[
\forall n \geq N ~~ \left|\int s_\e ~d\mu_n - \int s_\e ~d\mu\right| < \e.
\]
Then
\begin{align*}
\left|\int f ~d\mu_n - \int f ~d\mu\right| &\leq \left|\int (f - s_\e) ~d\mu_n\right|\\
                                           &+ \left|\int s_\e ~d\mu_n - \int s_\e ~d\mu\right| + \left|\int (s_\e - f) ~d\mu\right| < 3\e
\end{align*}
for all integers $n \geq N$, and the result follows.
\end{proof}

We end this section with an important result due to Andrei Kolmogorov and Yuri Prokhorov whose proof is omitted, but is available in a book by Patrick Billingsley \cite[p.\ 17]{copm}.

\begin{theo}\label{kolmogorov_prokhorov}
Let $X$ be a metric space; let $\mu$ and $\mu_n$ for all positive integers $n$ be measures on $(X,\mc{B}(X))$. Suppose that $\mc{A} \subseteq \mc{B}(X)$ such that
\begin{enumerate}
\item[(i)] $\mc{A}$ is closed under finite intersections, and

\item[(ii)] every open subset of $X$ is the union of countably many elements of $\mc{A}$.
\end{enumerate}

If $\mu_n(A) \to \mu(A)$ for all $A \in \mc{A}$, then $(\mu_n)_{n=1}^\infty$ converges weakly to $\mu$.
\end{theo}
%
%
\section{Rooted and birooted graphs}

Next we look at some more basic concepts, which are more specific to our purposes. In the remaining sections, we fix a positive integer $\Delta$.

Let $G$ be a graph. Denote by $G_x$ the connected component of $G$ whose vertex set contains $x$. Define $d_G(x,y)$ to be the length of the shortest path from $x$ to $y$ in $G$ if $G$ is connected. For every $r \in \NN$ and $o \in V(G)$, $B_G(o,r)$ is the subgraph of $G$ induced by the set of vertices
\[
\{x \in V(G) ~:~ d_{G_o}(o,x) \leq r\},
\]
and $N_G(o)$ is the set of vertices that are adjacent to $o$.

\begin{defn}
A \emph{rooted graph} is a pair $(G,o)$ where $G$ is a graph and $o \in V(G)$; a \emph{birooted graph} is a triple $(G,o_1,o_2)$ where $G$ is a graph, $o_1 \in V(G)$, and $o_2 \in N_G(o_1)$.
\end{defn}

Let $\Gr$ be the set of all isomorphism classes $[G,o]$ of countable, connected rooted graphs $(G,o)$ such that $\deg_G(x) \leq \Delta$ for all $x \in V(G)$.

Define the metric $\rho : \Gr \times \Gr \to \RR$ as follows:
\[
\rho([G,o],[H,p]) =
\begin{cases}
0 & \text{ if $[G,o] = [H,p]$,}\\
2^{-r} & \text{ otherwise}
\end{cases}
\]
where $r = \sup\{s \in \NN ~:~ [B_G(o,s),o] = [B_H(p,s),p]\}$. Denote by $\tau$ the topology induced by the metric $\rho$. That is, a basis for $\tau$ is the collection of balls in the metric space $(\Gr,\rho)$.

Similarly, $\vec\Gr$ is the set of all isomorphism classes $[G,o_1,o_2]$ of countable, connected birooted graphs $(G,o_1,o_2)$ such that $\deg_G(x) \leq \Delta$ for all $x \in V(G)$.

An analogous metric $\vec\rho : \vec\Gr \times \vec\Gr \to \RR$ is defined by
\[
\vec\rho([G,o_1,o_2],[H,p_1,p_2]) =
\begin{cases}
0 & \text{ if $[G,o_1,o_2] = [H,p_1,p_2]$,}\\
2^{-r} & \text{ otherwise}
\end{cases}
\]
where $r = \sup\{s \in \NN ~:~ [B_G(o_1,s),o_1,o_2] = [B_H(p_1,s),p_1,p_2]\}$. Unsurprisingly, the topology induced by $\vec\rho$ is denoted by $\vec\tau$.

Of course, the reader should not believe that $\rho$ and $\vec\rho$ are, in fact, ultrametrics without careful verification. However, rather than restate the arguments here, we refer the reader to this author's previous work \cite{wcolofg}.

\begin{theo}
The pairs $(\Gr,\rho)$ and $(\vec\Gr,\vec\rho)$ are compact ultrametric spaces.
\end{theo}

We now turn our attention to another collection of graphs, this time having no specified root. Let $\Graph$ be the set of all isomorphism classes of finite graphs $G$ such that $\deg_G(x) \leq \Delta$ for all $x \in V(G)$.

\begin{defn}
A \emph{rooted $r$-ball} is a rooted graph $[G,o] \in \Gr$ such that $d_G(x,o) \leq r$ for all $x \in V(G)$. The set of rooted $r$-balls is denoted by $U_r$. If $G \in \Graph$ and $o \in V(G)$, then $[B_G(o,r),o] \in U_r$ is the \emph{rooted $r$-ball around $o$ in $G$}.

A \emph{birooted $r$-ball} is a birooted graph $[G,o_1,o_2] \in \vec\Gr$ such that $[G,o_1] \in U_r$. The set of birooted $r$-balls is denoted by $\vec{U}_r$.
\end{defn}

If $\alpha \in U_r$ and $\vec\alpha \in \vec{U}_r$, let
\[
T_r(\Gr,\alpha) = \{[G,o] \in \Gr ~:~ [B_G(o,r),o] = \alpha\}
\]
and
\[
T_r(\vec\Gr,\vec\alpha) = \{[G,o_1,o_2] \in \vec\Gr ~:~ [B_G(o_1,r),o_1,o_2] = \vec\alpha\}.
\]
The strange notation of the collections above is adopted from papers by G\'abor Elek \cite{nolofg,otlolggs}, although with the addition of a subscript on the $T$ for better clarity.

Following a few technical results, it will be shown that the two collections above are important subsets of $\Gr$ and $\vec\Gr$.

\begin{prop}\label{iso_is_iso}
Graph isomorphisms are isometries.
\end{prop}

\begin{proof}
Let $\varphi : G \to H$ be a graph isomorphism for some graphs $G$ and $H$. If $x$ and $y$ are connected by a shortest path $P$ in $G$, then $\varphi(x)$ and $\varphi(y)$ are connected by the shortest path $\varphi(P)$, and so
\[
d_G(x,y) = |E(P)| = |E(\varphi(P))| = d_H(\varphi(x),\varphi(y)).
\]
Hence $\varphi$ preserves the shortest path metric, meaning it is an isometry.
\end{proof}

\begin{lem}\label{equal_balls_bounded_dist}
If $[G,o]$ and $[H,p]$ are distinct, then
\[
[B_G(o,r),o] = [B_H(p,r),p]
\]
if and only if
\[
\rho([G,o],[H,p]) \leq 2^{-r}.
\]
\end{lem}

\begin{proof}
Let $\rho([G,o],[H,p]) = 2^{-s}$ where
\[
s = \sup\{t \in \NN ~:~ [B_G(o,t),o] = [B_H(p,t),p]\}.
\]
If $[B_G(o,r),o] = [B_H(p,r),p]$, then $r \leq s$, and so $2^{-s} \leq 2^{-r}$. Conversely, assume that $2^{-s} \leq 2^{-r}$. That is, $r \leq s$. By definition, $[B_G(o,s),o] = [B_H(p,s),p]$. Let $\varphi : B_G(o,s) \to B_H(p,s)$ be a graph isomorphism such that $p = \varphi(o)$. Note that $B_G(o,r) \subseteq B_G(o,s)$, and consider the restriction $\varphi'$ of $\varphi$ to $B_G(o,r)$. The image of $B_G(o,r)$ under $\varphi'$ is $B_H(p,r)$ because $\varphi$ is an isometry.
\end{proof}

\begin{prop}\label{elek_top_is_dist_top}
The following equalities hold:
\[
T_r(\Gr,[B_G(o,r),o]) = B_\rho([G,o],2^{-r})
\]
and
\[
T_r(\vec\Gr,[B_G(o_1,r),o_1,o_2]) = B_{\vec\rho}([G,o_1,o_2],2^{-r}).
\]
\end{prop}

\begin{proof}
To see that the first equality is true, observe that
\begin{align*}
[H,p] \in T_r(\Gr,[B_G(o,r),o]) ~~ &\Leftrightarrow ~~ [B_H(p,r),p] = [B_G(o,r),o]\\
                                   &\Leftrightarrow ~~ \rho([G,o],[H,p]) \leq 2^{-r}\\
                                   &\Leftrightarrow ~~ [H,p] \in B_\rho([G,o],2^{-r}),
\end{align*}
where the second equivalence holds by Lemma \ref{equal_balls_bounded_dist}. The proof of the second equality is analogous.
\end{proof}

\begin{cor}\label{t_sets_are_bases}
The collections $\{T_r(\Gr,\alpha) ~:~ r \in \NN ~;~ \alpha \in U_r\}$ and $\{T_r(\vec\Gr,\vec\alpha) ~:~ r \in \NN ~;~ \vec\alpha \in \vec{U}_r\}$ are bases for the topologies $\tau$ and $\vec\tau$, respectively.
\end{cor}

\begin{proof}
Let $\alpha \in U_r$. Since $\alpha = [G,o]$ for some $[G,o] \in \Gr$ and $d_G(x,o) \leq r$ for all $x \in V(G)$, it follows that $\alpha = [B_G(o,r),o]$. That is,
\begin{align*}
\{T_r(\Gr,\alpha) ~:~ &r \in \NN ~;~ \alpha \in U_r\}\\
                      &= \{T_r(\Gr,[B_G(o,r),o]) ~:~ r \in \NN ~;~ [G,o] \in \Gr\}\\
                      &= \{B_\rho([G,o],2^{-r}) ~:~ r \in \NN ~;~ [G,o] \in \Gr\}
\end{align*}
where the second equality holds by Proposition \ref{elek_top_is_dist_top}. The same is true for the latter collection.
\end{proof}

\begin{cor}\label{t_sets_are_closed_open}
The sets $T_r(\Gr,\alpha)$ and $T_r(\vec\Gr,\vec\alpha)$ are both closed and open in $\Gr$ and $\vec\Gr$, respectively.
\end{cor}

\begin{proof}
Since $\Gr$ and $\vec\Gr$ are ultrametric spaces, the result is true by Corollary \ref{balls_are_closed_open}.
\end{proof}

\begin{prop}\label{t_satisfies_kp_conditions}
The collection
\[
\{T_r(\vec\Gr,\vec\alpha) ~:~ r \in \NN ~;~ \vec\alpha \in \vec{U}_r\} \cup \{\emptyset\}
\]
\begin{enumerate}
\item[(i)] is closed under finite intersections, and

\item[(ii)] every open subset of $\vec\Gr$ is a finite union of its elements.
\end{enumerate}
\end{prop}

\begin{proof}
The result easily follows from Corollary \ref{t_sets_are_bases} and the compactness of $\vec\Gr$.
\end{proof}
%
%
\section{Laws}

\begin{defn}
The \emph{law} is a function $\Psi : \Graph \to \mc{M}(\Gr)$ defined as follows: for every graph $G \in \Graph$,
\[
\Psi(G)[G_o,o] = \frac{|\Aut(G)o|}{|V(G)|}
\]
if $G_o$ is a connected component of $G$ for some $o \in V(G)$, and $\Psi(G) = 0$ elsewhere. Here $\Aut(G)$ is the group of automorphisms on $G$, and $\Aut(G)o$ is the \emph{orbit} of the vertex $o$ in $G$:
\[
\Aut(G)o = \{x \in V(G) ~:~ \exists \varphi \in \Aut(G) ~~ \varphi(x) = o\}.
\]
The image $\Psi(G)$ of a finite graph $G \in \Graph$ is a probability measure on $\Gr$ called \emph{the law of $G$}. Usually, we will simply write \emph{the law} when no reference to a specific graph is necessary.
\end{defn}

If $\alpha \in U_r$ and $G \in \Graph$, let
\[
T_r(G,\alpha) = \{x \in V(G) ~:~ [B_G(x,r),x] = \alpha\}
\]
and
\[
p_G(\alpha,r) = \frac{|T_r(G,\alpha)|}{|V(G)|}.
\]

Using this notation, G\'abor Elek \cite{nolofg,otlolggs} defines the weak convergence of ``laws'' in the following way.

\begin{defn}
A graph sequence $(G_n)_{n=1}^\infty$ in $\Graph$ \emph{converges weakly} if there is a measure $\mu$ on $\Gr$ such that
\[
p_{G_n}(\alpha,r) \to \mu(T_r(\Gr,\alpha))
\]
for all $r \in \NN$ and $\alpha \in U_r$.
\end{defn}

To see that the quotation marks around the word ``laws'' are not necessary, consider this next pair of results.

\begin{lem}\label{law_eval_on_set}
Suppose that $G \in \Graph$. Then
\[
\Psi(G)(T_r(\Gr,\alpha)) = \frac{|T_r(G,\alpha)|}{|V(G)|}
\]
for all $r \in \NN$ and $\alpha \in U_r$.
\end{lem}

\begin{proof}
If $G \in \Graph$, then
\[
\Psi(G)(T_r(\Gr,\alpha)) = \int \chi_{T_r(\Gr,\alpha)} ~d\Psi(G) = \frac{1}{|V(G)|}\sum_{x \in V(G)} \chi_{T_r(\Gr,\alpha)}[G,x] = \frac{|T_r(G,\alpha)|}{|V(G)|}
\]
for all $r \in \NN$ and $\alpha \in U_r$ where the third equality holds because $\chi_{T_r(\Gr,\alpha)}[G,o] = 1$ precisely when $\chi_{T_r(G,\alpha)}(o) = 1$ for all $o \in V(G)$.
\end{proof}

\begin{prop}
Let $G_n \in \Graph$ for all positive integers $n$. The sequence of laws $(\Psi(G_n))_{n=1}^\infty$ converges weakly if and only if the graph sequence $(G_n)_{n=1}^\infty$ does too.
\end{prop}

\begin{proof}
Suppose that $(\Psi(G_n))_{n=1}^\infty$ converges weakly to some measure $\mu$ on $\Gr$. By Corollary \ref{t_sets_are_closed_open}, $T_r(\Gr,\alpha)$ is closed and open, which means its characteristic function is continuous on $\Gr$. Using the definition of weak convergence and Lemma \ref{law_eval_on_set},
\[
\frac{|T_r(G_n,\alpha)|}{|V(G_n)|} = \int \chi_{T_r(\Gr,\alpha)} ~d\Psi(G_n) \rightarrow \int \chi_{T_r(\Gr,\alpha)} ~d\mu = \mu(T_r(\Gr,\alpha)). \tag{$\star$}
\]
for all $r \in \NN$ and $\alpha \in U_r$. Hence $(G_n)_{n=1}^\infty$ converges weakly.

Conversely, assume that $(G_n)_{n=1}^\infty$ converges weakly. Then ($\star$) holds for all $r \in \NN$ and $\alpha \in U_r$. Hence the sequence $(\Psi(G_n))_{n=1}^\infty$ converges weakly to $\mu$ by Theorem \ref{weak_conv_iff_char_func_conv} and Proposition \ref{elek_top_is_dist_top}.
\end{proof}
%
%
\section{Unimodularity versus involution invariance}

The following section guides the reader to the first of our goals. Namely, a proof that weak limits preserve the concept known as unimodularity. This result was stated by Itai Benjamini and Oded Schramm \cite[p.\ 10]{rodlofpg}, but we give a detailed argument.

%
%
\subsection{Preliminaries}

\begin{defn}
A measure $\mu$ on $\Gr$ is \emph{unimodular} if
\[
\int \sum_{x \in N_G(o)} f[G,x,o]~d\mu[G,o] = \int \sum_{x \in N_G(o)} f[G,o,x]~d\mu[G,o]
\]
for all nonnegative real-valued Borel functions $f$ on $\vec\Gr$.
\end{defn}

Define the function $\iota : \vec\Gr \to \vec\Gr$ by $\iota[G,x,y] = [G,y,x]$ for all $[G,x,y] \in \vec\Gr$. Every Borel subset $A$ of $\vec\Gr$ induces a function $f_A : \Gr \to \NN$ defined by
\[
f_A[G,o] = |\{x \in N_G(o) ~:~ [G,o,x] \in A\}|
\]
for all $[G,o] \in \Gr$. Let $\mu$ be a measure on $\Gr$. The measure $\vec\mu$ on $\vec\Gr$ is defined by $\vec\mu(A) = \int f_A ~d\mu$ for all Borel subsets $A$ of $\vec\Gr$.

\begin{defn}\label{defn_of_vec}
A measure $\mu$ on $\Gr$ is \emph{involution invariant} if $\iota_\ast(\vec\mu) = \vec\mu$.
\end{defn}

In fact, the concepts of unimodularity and involution invariance are logically equivalent as the following theorem demonstrates. This result seems to be known based on the different, yet equivalent, approaches taken by Elek \cite{nolofg,otlolggs}, and Aldous and Lyons \cite{pourn}, but there is no explicit argument in the literature.

\begin{theo}\label{unimodular_iff_inv_inv}
A measure $\mu$ on $\Gr$ is unimodular if and only if it is involution invariant.
\end{theo}

\begin{proof}
Note that
\[
\sum_{x \in N_G(o)} \chi_A[G,o,x] = |\{x \in N_G(o) ~:~ [G,o,x] \in A\}| = f_A[G,o]
\]
and
\[
\sum_{x \in N_G(o)} (\chi_A \circ \iota)[G,o,x] = |\{x \in N_G(o) ~:~ \iota[G,o,x] \in A\}| = f_{\iota(A)}[G,o]
\]
for all Borel subsets $A$ of $\vec\Gr$. Suppose that $\mu$ is unimodular. Then
\[
\iota_\ast(\vec\mu)(A) = \vec\mu(\iota(A)) = \int f_{\iota(A)}[G,o] ~d\mu[G,o] = \int f_A[G,o] ~d\mu[G,o] = \vec\mu(A)
\]
for all Borel subsets $A$ of $\vec\Gr$. Conversely, if $\iota_\ast(\vec\mu) = \vec\mu$, then
\[
\int \sum_{x \in N_G(o)} \chi_A[G,o,x] ~d\mu[G,o] = \int \sum_{x \in N_G(o)} (\chi_A \circ \iota)[G,o,x] ~d\mu[G,o]
\]
for all Borel subsets $A$ of $\vec\Gr$. Since this holds for all characteristic functions, it is true for all simple functions, and so for all nonnegative real-valued Borel functions.
\end{proof}

%
%
\subsection{Weak limits preserve unimodularity}

Having defined and reconciled the definitions of unimodularity and involution invariance, it is time to overcome several technical results, and accomplish our first goal.

\begin{lem}\label{equal_balls_equiv_nbhds}
If $\varphi : (B_G(o,r),o) \to (B_H(p,r),p)$ is a rooted graph isomorphism, then $N_H(p) = \varphi(N_G(o))$.
\end{lem}

\begin{proof}
If $y \in \varphi(N_G(o))$, then $y = \varphi(x)$ for some $x \in N_G(o)$. Since $\varphi$ is a graph isomorphism, Proposition \ref{iso_is_iso} implies that
\[
d_H(y,p) = d_H(\varphi(x),\varphi(o)) = d_G(x,o) = 1,
\]
and so $y \in N_H(p)$. Thus $\varphi(N_G(o)) \subseteq N_H(p)$. On the other hand, assume that $y \in N_H(p)$. Since $\varphi$ is bijective, there is an $x \in V(G)$ such that $y = \varphi(x)$. Furthermore,
\[
d_G(x,o) = d_H(\varphi(x),\varphi(o)) = d_H(y,p) = 1,
\]
which means $x \in N_G(o)$.
\end{proof}

\begin{prop}\label{f_on_t_is_cont}
The function $f_A$ is Lipschitz when $A = T_r(\vec\Gr,\vec\alpha)$. In particular, it is continuous.
\end{prop}

\begin{proof}
Let $A = T_r(\vec\Gr,\vec\alpha)$. If $\rho([G,o],[H,p]) \leq 2^{-r}$, there is a rooted graph isomorphism $\varphi : (B_G(o,r),o) \to (B_H(p,r),p)$. By Lemma \ref{equal_balls_equiv_nbhds}, $N_H(p) = \varphi(N_G(o))$. Since $\varphi$ is an isomorphism, it is easy to see that $f_A[H,p] = f_A[G,o]$. On the other hand, assume that $\rho([G,o],[H,p]) > 2^{-r}$. Then
\[
|f_A[G,o] - f_A[H,p]| \leq \Delta = \Delta2^r2^{-r} < \Delta2^r \cdot \rho([G,o],[H,p])
\]
because $0 \leq f_A \leq \Delta$. Hence $f_A$ is $\Delta2^r$-Lipschitz, and so it is continuous.
\end{proof}

\begin{lem}\label{adj_implies_subset}
If $y \in N_G(x)$, then $B_G(y,r - 1) \subseteq B_G(x,r)$.
\end{lem}

\begin{proof}
Suppose that $y \in N_G(x)$ and $z \in B_G(y,r - 1)$. Then $d_G(x,y) = 1$ and $d_G(y,z) \leq r - 1$, so
\[
d_G(x,z) \leq d_G(x,y) + d_G(y,z) = 1 + d_G(y,z) \leq r,
\]
which means $z \in B_G(x,r)$.
\end{proof}

\begin{lem}\label{iso_restricts_to_iso}
If $[B_G(o_1,r),o_1,o_2] = [B_H(p_1,r),p_1,p_2]$, then
\[
[B_G(o_2,r - 1),o_2,o_1] = [B_H(p_2,r - 1),p_2,p_1].
\]
\end{lem}

\begin{proof}
Suppose that $[B_G(o_1,r),o_1,o_2] = [B_H(p_1,r),p_1,p_2]$. There is a graph isomorphism $\varphi : B_G(o_1,r) \to B_H(p_1,r)$. By Lemma \ref{adj_implies_subset}, $B_G(o_2,r - 1) \subseteq B_G(o_1,r)$. Let $\varphi'$ be the restriction of $\varphi$ to $B_G(o_2,r - 1)$. The image of $\varphi'$ is $B_H(p_2,r - 1)$ because $\varphi$ is an isometry. It follows that $\varphi' : B_G(o_2,r - 1) \to B_H(p_2,r - 1)$ is a graph isomorphism. Furthermore, $\varphi'(o_1) = \varphi(o_1) = p_1$ and $\varphi'(o_2) = \varphi(o_2) = p_2$.
\end{proof}

\begin{prop}\label{inv_is_cont}
The function $\iota$ is a continuous involution. In fact, $\iota$ is a self-homeomorphism of $\vec\Gr$.
\end{prop}

\begin{proof}
If $[G,o_1,o_2],[H,p_1,p_2] \in \vec\Gr$ are distinct, then
\[
\vec\rho([G,o_1,o_2],[H,p_1,p_2]) = 2^{-r}
\]
and $[B_G(o_1,r),o_1,o_2] = [B_H(p_1,r),p_1,p_2]$. By Lemma \ref{iso_restricts_to_iso},
\[
[B_G(o_2,r - 1),o_2,o_1] = [B_H(p_2,r - 1),p_2,p_1],
\]
and so
\[
\vec\rho(\iota[G,o_1,o_2],\iota[H,p_1,p_2]) = \vec\rho([G,o_2,o_1],[H,p_2,p_1]) \leq 2^{-(r-1)} = 2 \cdot 2^{-r}.
\]
Hence $\iota$ is $2$-Lipschitz, and so it is continuous. Furthermore, $\iota$ is a self-homeomorphism because it is an involution.
\end{proof}

\begin{prop}\label{mu_conv_mu_arrow_conv}
If $(\mu_n)_{n=1}^\infty$ converges weakly to $\mu$, then $(\vec\mu_n)_{n=1}^\infty$ converges weakly to $\vec\mu$.
\end{prop}

\begin{proof}
By Proposition \ref{t_satisfies_kp_conditions} and Theorem \ref{kolmogorov_prokhorov}, it suffices to show that
\[
\vec\mu_n(T_r(\vec\Gr,\vec\alpha)) \to \vec\mu(T_r(\vec\Gr,\vec\alpha))
\]
for all $r \in \NN$ and $\vec\alpha \in \vec{U}_r$. By Proposition \ref{f_on_t_is_cont}, $f_A$ is continuous when $A = T_r(\vec\Gr,\vec\alpha)$. Then
\[
\vec\mu_n(A) = \int f_A ~d\mu_n \to \int f_A ~d\mu = \vec\mu(A)
\]
because $(\mu_n)_{n=1}^\infty$ converges weakly to $\mu$. Thus $(\vec\mu_n)_{n=1}^\infty$ converges weakly to $\vec\mu$.
\end{proof}

Finally, we arrive at our first main result. Using the technical propositions stated above, we proceed to demonstrate the following. The idea for the proof of the following theorem is due to a paper by David Aldous and J. Michael Steele \cite[p.\ 40]{tom}.

\begin{theo}\label{limits_preserve_inv_inv}
If $(\mu_n)_{n=1}^\infty$ is a sequence of involution invariant measures on $\Gr$ that converges weakly to a measure $\mu$ on $\Gr$, then $\mu$ is involution invariant.
\end{theo}

\begin{proof}
For convenience, let $\mu = \lim_{n \to \infty} \mu_n$. By Proposition \ref{mu_conv_mu_arrow_conv}, $\vec\mu = \lim_{n \to \infty} \vec\mu_n$. Using Proposition \ref{inv_is_cont} with Proposition \ref{pushforward_is_continuous}, we see that $\iota_\ast(\vec\mu) = \lim_{n \to \infty} \iota_\ast(\vec\mu_n)$. Since $\mu_n$ is involution invariant for all positive integers $n$, it follows that $\iota_\ast(\vec\mu) = \lim_{n \to \infty} \vec\mu_n$, and so $\iota_\ast(\vec\mu) = \vec\mu$, which means $\mu$ is involution invariant.
\end{proof}

\begin{cor}\label{limits_preserve_unimodularity}
If $(\mu_n)_{n=1}^\infty$ is a sequence of unimodular measures on $\Gr$ that converges weakly to a measure $\mu$ on $\Gr$, then $\mu$ is unimodular.
\end{cor}

\begin{proof}
This follows immediately by Theorem \ref{unimodular_iff_inv_inv}.
\end{proof}
%
%
\section{Graphings}
In this section, the primary focus will be on discovering a potentially vast new source of examples of unimodular measures by showing that the law of a graphing is unimodular. Before doing so, the reader needs to know what a graphing is.

%
%
\subsection{Preliminaries}
For the purposes of this article, we will be using G\'abor Elek's definition of a graphing \cite{nolofg}. Although, as it is later shown, there is a more general notion.

\begin{defn}
Let $\mu$ be a measure on a Borel space $X$. A \emph{measurable graphing} is a tuple $\GG = (X,i_1,i_2,\ldots,i_k,\mu)$ where $i_j$ is a measure-preserving Borel involution of $X$ for each $j \in \{1,2,\ldots,k\}$.

The measurable graphing $\GG$ determines an equivalence relation $\sim_\GG$ on $X$ defined as follows: $x \sim_\GG y$ if and only if there is a subset $\{x_1,x_2,\ldots,x_m\} \subseteq X$ such that
\begin{enumerate}
\item[(i)] $x_1 = x$ and $x_m = y$, and

\item[(ii)] for each $i \in \{1,2,\ldots,m - 1\}$, there is a $j \in \{1,2,\ldots,k\}$ such that $x_{i+1} = i_j(x_i)$
\end{enumerate}
for all $(x,y) \in X \times X$. The \emph{leafgraph} of $\GG$ is a graph $\LL$ whose vertex set is $X$, and $x$ is adjacent to $y$ in $\LL$ precisely when $y = i_j(x)$ for some $j \in \{1,2,\ldots,k\}$.
\end{defn}

In passing, we mention the following straightforward fact that relates the equivalence relation $\sim_\GG$ to the leafgraph of $\GG$.

\begin{prop}
If $\GG = (X,i_1,i_2,\ldots,i_k,\mu)$ is a measurable graphing, then the equivalence classes of $\sim_\GG$ are the connected components of $\LL$. Specifically, $V(\LL_x) = [x]_{\sim_\GG}$ and $E(\LL_x) = \{yz ~:~ \exists j \in \{1,2,\ldots,k\} ~~ i_j(y) = z\}$ for all $x \in X$.
\end{prop}

Next we define the law of a graphing, which is similar to the law of a finite graph seen previously.

\begin{defn}
Let $\GG = (X,i_1,i_2,\ldots,i_k,\mu)$ be a measurable graphing. Denote by $\LL$ the leafgraph of $\GG$. The \emph{law of $\GG$} is the probability measure $\Psi(\GG)$ on $\Gr$ defined by
\[
\Psi(\GG)(T_r(\Gr,\alpha)) = \mu(\{x \in X ~:~ [B_\LL(x,r),x] = \alpha\})
\]
for all $r \in \NN$ and $\alpha \in U_r$.
\end{defn}

By writing $\Psi(\LL)$ instead of $\Psi(\GG)$, this definition of a law expands the domain of the function $\Psi$ to include all leafgraphs. However, we will opt to use $\Psi(\GG)$ instead.

The next proposition demonstrates why the definition of the law of a graphing is consistent with that of the law of a finite graph.

\begin{prop}
The law $\Psi(G)$ of a graph $G \in \Graph$ is the law of the measurable graphing $\GG = (V(G),\{i_{xy} ~:~ xy \in E(G)\},\mu)$ where $\mu$ is the uniform measure on $V(G)$, and $i_{xy} : V(G) \to V(G)$ maps $x$ to $y$, $y$ to $x$, and fixes the other vertices.
\end{prop}

\begin{proof}
Since $\Aut(G)$ partitions the vertex set of $G$, $V(G) = \bigsqcup_{j=1}^k \Aut(G)j$. Furthermore, $|\Aut(G)j| \cdot \chi_A[G_j,j] = |\{x \in \Aut(G)j ~:~ [G_j,j] \in A\}|$, and $[G_x,x] = [G_j,j]$ because $x$ and $j$ are in the same orbit. Then
\begin{align*}
\Psi(G)(A) &= \sum_{j=1}^k \frac{|\Aut(G)j| \cdot \chi_A[G_j,j]}{|V(G)|}\\
           &= \sum_{j=1}^k \frac{|\{x \in \Aut(G)j ~:~ [G_x,x] \in A\}|}{|V(G)|}\\
           &= \frac{\left|\bigsqcup_{j=1}^k \{x \in \Aut(G)j ~:~ [G_x,x] \in A\}\right|}{|V(G)|}\\
           &= \frac{|\{x \in V(G) ~:~ [G_x,x] \in A\}|}{|V(G)|}
\end{align*}
for all Borel subsets $A$ of $\Gr$. In particular,
\[
\Psi(G)(T_r(\Gr,\alpha)) = \mu(\{x \in V(G) ~:~ [B_{G_x}(x,r),x] = \alpha\})
\]
for all $r \in \NN$ and $\alpha \in U_r$. Hence $\Psi(G) = \Psi(\GG)$.
\end{proof}

To bridge the gap between the law of $\GG$ and $\mu$, the reader is encouraged to study the following proposition, which links the two measures.

\begin{prop}\label{law_to_mu}
Let $\GG = (X,i_1,i_2,\ldots,i_k,\mu)$ be a measurable graphing. If $g : \Gr \to \RR$ is a Borel function, then
\[
\int_\Gr g ~d\Psi(\GG) = \int_X g[\LL_x,x] ~d\mu(x).
\]
\end{prop}

\begin{proof}
Define the function $q : (X,\mu) \to (\Gr,\Psi(\GG))$ by $q(x) = [\LL_x,x]$ for all $x \in X$. Observe that
\[
q(y) \in T_r(\Gr,\alpha) ~~ \Leftrightarrow ~~ [\LL_y,y] \in T_r(\Gr,\alpha) ~~ \Leftrightarrow ~~ [B_{\LL_y}(y,r),y] = \alpha,
\]
which means
\[
q^{-1}(T_r(\Gr,\alpha)) = \{x \in X ~:~ [B_{\LL_x}(x,r),x] = \alpha\}.
\]
Furthermore, $B_\LL(x,r) = B_{\LL_x}(x,r)$. Then
\[
\mu(q^{-1}(T_r(\Gr,\alpha))) = \mu(\{x \in X ~:~ [B_{\LL_x}(x,r),x] = \alpha\}) = \Psi(\GG)(T_r(\Gr,\alpha))
\]
for all $r \in \NN$ and $\alpha \in U_r$, and so $q_\ast(\mu) = \Psi(\GG)$. Hence
\[
\int_\Gr g ~d\Psi(\GG) = \int_\Gr g ~dq_\ast(\mu) = \int_X (g \circ q) ~d\mu,
\]
as required.
\end{proof}

Although the following result was shown before in this author's Honours project \cite{wcolofg}, the following argument presents another, more suitable, viewpoint.

\begin{prop}
If $G \in \Graph$, then $\Psi(G)$ is unimodular.
\end{prop}

\begin{proof}
Define the relation $S = \{(x,y) \in V(G) \times V(G) ~:~ xy \in E(G)\}$. Observe that $S$ is symmetric; that is, $(x,y) \in S$ if and only if $(y,x) \in S$. Furthermore, $xy \in E(G)$ if and only if $y \in N_G(x)$, and $N_G(x) = N_{G_x}(x)$. Then
\begin{align*}
\int \sum_{y \in N_H(x)} f[H,x,y] ~d\Psi(G)[H,x] &= \frac{1}{|V(G)|} \sum_{x \in V(G)} \sum_{y \in N_G(x)} f[G_x,x,y]\\
                                                 &= \frac{1}{|V(G)|} \sum_{(x,y) \in S} f[G_x,x,y]\\
                                                 &= \frac{1}{|V(G)|} \sum_{(y,x) \in S} f[G_x,x,y]\\
                                                 &= \frac{1}{|V(G)|} \sum_{(y,x) \in S} f[G_y,x,y]\\
                                                 &= \frac{1}{|V(G)|} \sum_{y \in V(G)} \sum_{x \in N_G(y)} f[G_y,x,y]\\
                                                 &= \int \sum_{x \in N_H(y)} f[H,x,y] ~d\Psi(G)[H,y]
\end{align*}
where the fourth equality holds because $G_x = G_y$ whenever $x$ is adjacent to $y$.
\end{proof}

Consider the measurable graphing $\GG = (X,i_1,i_2,\ldots,i_k,\mu)$ whose leafgraph is $\LL$. Let $S_\LL = \{(x,y) \in X \times X ~:~ xy \in E(\LL)\}$ and
\[
\vec\mu(B) = \int_X |\{y \in N_\LL(x) ~:~ (x,y) \in B\}| ~d\mu(x)
\]
for all Borel subsets $B$ of $S_\LL$. For convenience, we will use $\iota$ to denote two different functions. The reader should already be familiar with the first of these functions from Definition \ref{defn_of_vec}. Let
\begin{align*}
\iota : \vec\Gr &\to \vec\Gr\\
        [G,x,y] &\mapsto [G,y,x]
\end{align*}
and
\begin{align*}
\iota : S_\LL &\to S_\LL\\
        (x,y) &\mapsto (y,x),
\end{align*}
which are both involutions. As for measures on $\Gr$, there is a similar notion of involution invariance for measures on $X$.

\begin{defn}
A measure $\mu$ on $X$ is \emph{involution invariant} if $\iota_\ast(\vec\mu) = \vec\mu$.
\end{defn}

%
%
\subsection{Laws of graphings are unimodular}
With the basic tools in hand, we may now construct a proof that laws of graphings, when dealing with unimodularity, behave in the same way as laws of finite graphs.

For the remainder of this section, let $B_A = \{(x,y) \in S_\LL ~:~ [\LL_x,x,y] \in A\}$ and $A_B = \{[G,x,y] \in \vec\Gr ~:~ (x,y) \in B\}$ for all Borel subsets $A$ of $\vec\Gr$ and $B$ of $S_\LL$.

\begin{prop}\label{directed_edge_facts}
If $A$ is a Borel subset of $\vec\Gr$, then $\chi_A[\LL_x,x,y] = \chi_{B_A}(x,y)$ for all $(x,y) \in S_\LL$. If $B$ is a Borel subset of $S_\LL$, then $\chi_B(x,y) = \chi_{A_B}[\LL_x,x,y]$ for all $(x,y) \in S_\LL$. Furthermore,
\[
f_A[\LL_x,x] := |\{y \in N_\LL(x) ~:~ [\LL_x,x,y] \in A\}| = \sum_{y \in N_\LL(x)} \chi_A[\LL_x,x,y]
\]
and
\[
|\{y \in N_\LL(x) ~:~ (x,y) \in B\}| = \sum_{y \in N_\LL(x)} \chi_B(x,y)
\]
for all Borel subsets $A$ of $\vec\Gr$ and $B$ of $S_\LL$.
\end{prop}

\begin{lem}\label{vec_mu_is_int_sum}
If $B$ is a Borel subset of $S_\LL$, then
\[
\vec\mu(B) = \int_X \sum_{y \in N_\LL(x)} \chi_B(x,y) ~d\mu(x).
\]
\end{lem}

\begin{proof}
By Proposition \ref{directed_edge_facts},
\[
\int_X |\{y \in N_\LL(x) ~:~ (x,y) \in B\}| ~d\mu(x) = \int_X \sum_{y \in N_\LL(x)} \chi_B(x,y) ~d\mu(x),
\]
and the result follows.
\end{proof}

\begin{lem}\label{iota_preserves_edge_facts}
If $A$ and $B$ are Borel subsets of $\vec\Gr$ and $S_\LL$, respectively, then $A_{\iota(B)} = \iota(A_B)$ and $B_{\iota(A)} = \iota(B_A)$.
\end{lem}

\begin{theo}\label{law_is_inv_inv_iff_mu_is}
Let $\GG = (X,i_1,i_2,\ldots,i_k,\mu)$ be a measurable graphing. The law $\Psi(\GG)$ is involution invariant if and only if $\mu$ is involution invariant.
\end{theo}

\begin{proof}
Suppose that $\Psi(\GG)$ is involution invariant. Let $B$ be a Borel subset of $S_\LL$. Proposition \ref{directed_edge_facts}, Lemma \ref{vec_mu_is_int_sum}, and Lemma \ref{iota_preserves_edge_facts} tell us that
\[
\vec\mu(\iota(B)) = \int_\Gr \sum_{y \in N_\LL(x)} \chi_{\iota(A_B)}[\LL_x,x,y] ~d\Psi(\GG)[\LL_x,x],
\]
and we know that the right-hand side is equal to
\[
\int_\Gr |\{y \in N_\LL(x) ~:~ [\LL_x,x,y] \in \iota(A_B)\}| ~d\Psi(\GG)[\LL_x,x] = \vec\Psi(\GG)(\iota(A_B)),
\]
which means $\iota_\ast(\vec\mu)(B) = \vec\Psi(\GG)(\iota(A_B))$. A similar argument shows that $\vec\mu(B) = \vec\Psi(\GG)(A_B)$. Since $\Psi(\GG)$ is involution invariant, we see that $\iota_\ast(\vec\mu) = \vec\mu$.

Conversely, assume that $\mu$ is involution invariant. Let $A$ be a Borel subset of $\vec\Gr$. By Proposition \ref{law_to_mu},
\[
\vec\Psi(\GG)(\iota(A)) = \int_X f_{\iota(A)}[\LL_x,x] ~d\mu(x),
\]
and the right-hand side is equal to $\vec\mu(\iota(B_A))$ using Proposition \ref{directed_edge_facts}, Lemma \ref{vec_mu_is_int_sum}, and Lemma \ref{iota_preserves_edge_facts}. That is, $\vec\Psi(\GG)(\iota(A)) = \vec\mu(\iota(B_A))$. Analogously, $\vec\Psi(\GG)(A) = \vec\mu(B_A)$. Then $\iota_\ast(\vec\Psi(\GG)) = \vec\Psi(\GG)$ because $\mu$ is involution invariant.
\end{proof}

The question that remains is whether a measure $\mu$ from some measurable graphing is \emph{always} involution invariant. In fact, the answer to this question is affirmative. However, to prove this result, we consider a more general situation.

\begin{defn}
Let $X$ be a Borel space; let $E$ be a countable Borel equivalence relation on $X$. A \emph{general graphing} is a tuple $\GG = (X,\Gamma,\mu)$ where $\Gamma \subseteq E$ is an antireflexive and symmetric Borel relation.
\end{defn}

Intimately related to this type of graphing is the concept of invariance under equivalence relations, defined below, which is discussed more thoroughly in a set of lecture notes by Alexander Kechris and Benjamin Miller \cite{tioe}.

\begin{defn}
Let $X$ be a Borel space; let $E$ be a countable Borel equivalence relation on $X$. A measure $\mu$ on $X$ is \emph{$E$-invariant} if for all Borel bijections $f : A \to B$ where $A$ and $B$ are Borel subsets of $X$ and $\gr(f) \subseteq E$, we have $\mu(A) = \mu(B)$.
\end{defn}

Using the notation of Kechris and Miller, define the measures $M$ and $M'$ as follows:
\[
M(A) = \int_X |\{y \in X ~:~ (x,y) \in A\}| ~d\mu(x)
\]
and
\[
M'(A) = \int_X |\{y \in X ~:~ (y,x) \in A\}| ~d\mu(x)
\]
for all Borel subsets $A$ of $E$.

\begin{prop}
If $M = M'$, then $\mu$ is $E$-invariant.
\end{prop}

\begin{proof}
To show that $\mu$ is $E$-invariant, let $f : A \to B$ be a Borel bijection for some Borel subsets $A$ and $B$ of $X$. Note that $B = f(A)$ because $f$ is a bijection. Suppose that $\gr(f) \subseteq E$ where $\gr(f)$ is Borel. Then
\[
M(\gr(f)) = \int_A |\{y \in X ~:~ y = f(x)\}| ~d\mu(x) = \int_A |\{f(x)\}| ~d\mu(x) = \mu(A)
\]
and
\[
M'(\gr(f)) = \int_{f(A)} |\{y \in A ~:~ x = f(y)\}| ~d\mu(x) = \int_{f(A)} |\{f^{-1}(x)\}| ~d\mu(x) = \mu(f(A))
\]
are equal because $M = M'$ by assumption. That is, $\mu(A) = \mu(f(A)) = \mu(B)$.
\end{proof}

Kechris and Miller prove that the converse is also true \cite[p.\ 57]{tioe}, which leads to the following corollary.

\begin{cor}\label{E_inv_iff_M_eq_M_prime}
The measure $\mu$ is $E$-invariant if and only if $M = M'$.
\end{cor}

Now if the reader recalls, our definition of measurable graphing provides us with an equivalence relation, and this is precisely what we need to use Corollary \ref{E_inv_iff_M_eq_M_prime}.

\begin{theo}
Let $(X,\mu)$ be a measure space. If $\GG = (X,i_1,i_2,\ldots,i_k,\mu)$ is a measurable graphing, then $\mu$ is $\sim_\GG$-invariant.
\end{theo}

\begin{proof}
Let $\Gamma = \langle i_1,i_2,\ldots,i_k \rangle$ be the free group generated by the involutions. Denote by $E$ the equivalence relation $\sim_\GG$. The countable group $\Gamma$ acts on $X$ in a Borel fashion as follows: $\gamma \cdot x = \gamma(x)$ for all $x \in X$ and $\gamma \in \Gamma$. Note that $E = E_\Gamma^X$ where
\[
(x,y) \in E_\Gamma^X ~~\Leftrightarrow~~ \exists \gamma \in \Gamma ~~ \gamma \cdot x = y
\]
because any $\gamma \in \Gamma$ can be written as the composition of the generators $i_1,i_2,\ldots,i_k$.

Furthermore, $\mu$ is $\Gamma$-invariant: if $A$ is a Borel subset of $X$ and $\gamma \in \Gamma$, then
\[
\mu(\gamma(A)) = \mu(\gamma \cdot A) = \mu(A)
\]
because $(i_j)_\ast(\mu) = \mu$ and again $\gamma$ is a composition of $i_1,i_2,\ldots,i_k$. Using a proposition from the lectures notes by Kechris and Miller \cite[p.\ 57]{tioe}, we see that $\mu$ is $E$-invariant.
\end{proof}

\begin{cor}
The law of a measurable graphing $\GG = (X,i_1,i_2,\ldots,i_k,\mu)$ is unimodular.
\end{cor}

\begin{proof}
Note that $M\vert_{S_\LL} = \vec\mu$ and $M'\vert_{S_\LL} = \iota_\ast(\vec\mu)$, so $\iota_\ast(\vec\mu) = \vec\mu$. Theorem \ref{law_is_inv_inv_iff_mu_is} implies that the law $\Psi(\GG)$ of the measurable graphing $\GG$ is unimodular.
\end{proof}
%
%
\section{Open problems}

The purpose of this section is to acquaint the reader with several interesting questions that have yet to be resolved.

David Aldous and Russell Lyons \cite{pourn} asked the following in 2007, and it remains, in this author's eyes, one of the most important questions listed here.

\begin{oq}\label{uni_is_weak_limit}
Is every unimodular measure the weak limit of a sequence of laws of finite graphs?
\end{oq}

A related but weaker question is obtained by removing the finiteness condition.

\begin{oq}
Is every unimodular measure the weak limit of a sequence of laws of measurable graphings?
\end{oq}

The questions that follow, if true, combine to establish an affirmative answer to Open Question \ref{uni_is_weak_limit}.

\begin{oq}
Is every unimodular measure the law of some measurable graphing?
\end{oq}

\begin{oq}
Is the law of a measurable graphing the weak limit of a sequence of laws of finite graphs?
\end{oq}

We would also like to link the notion of unimodularity with that of sofic groups. Such groups were introduced by Mikhael Gromov and Benjamin Weiss. We refer the reader to a survey by Vladimir Pestov of the known and unknown results \cite{hasgabg}.

Recall that a Cayley graph of a group $\Gamma$ is the pair
\[
\Cay(\Gamma,S) = (\Gamma,\{(x,sx) \in \Gamma \times \Gamma ~:~ s \in S\})
\]
where $S$ is a set of generators of $\Gamma$.

\begin{defn}
A finitely generated group $\Gamma$ is \emph{sofic} if it has a finite symmetric set of generators $S$ such that for all positive real numbers $\e$ and $r \in \NN$, there is a finite directed graph $G = (V,E)$ edge-labelled by $S$, which has a finite subset of vertices $V_0 \subseteq V$ satisfying
\begin{enumerate}
\item[(i)] $\forall v \in V_0$, $B_G(v,r)$ is edge-labelled isomorphic to $B_{\Cay(\Gamma,S)}(1_\Gamma,r)$, and

\item[(ii)] $|V_0| \geq (1 - \e)|V|$.
\end{enumerate}
\end{defn}

\begin{oq}
Is a group sofic if and only if the Dirac measure on its Cayley graph is the weak limit of a sequence of laws of finite graphs?
\end{oq}

If we deviate from the general setting of rooted graphs introduced in this article, we may also consider expander graphs as the objects of study, as well as automorphism groups of graphs.
\nocite{*}
\bibliography{graphings_and_unimodularity}
\end{document}